\documentclass[a4paper,10pt,leqno]{amsart}
\usepackage{amsmath,amsfonts,amsthm,amssymb,dsfont}
\usepackage[alphabetic]{amsrefs}
\usepackage[OT4]{fontenc}
\usepackage{enumerate}
\usepackage{mathrsfs}
\usepackage{enumerate, xspace}
\usepackage[pdftex]{graphicx}
\usepackage{color}

\long\def\symbolfootnote[#1]#2{\begingroup
\def\thefootnote{\fnsymbol{footnote}}\footnote[#1]{#2}\endgroup}

\newtheorem{theorem}{Theorem}[section]
\newtheorem{lemma}[theorem]{Lemma}
\newtheorem{thm}[theorem]{Theorem}
\newtheorem{sublemma}[theorem]{Sublemma}
\newtheorem{prop}[theorem]{Proposition}
\newtheorem{cor}[theorem]{Corollary}

\theoremstyle{definition}
\newtheorem{rem}[theorem]{Remark}
\newtheorem{defin}[theorem]{Definition}
\newtheorem{ex}[theorem]{Example}
\newtheorem{quest}[theorem]{Question}

\renewcommand{\H}{\mathbf{H}}
\newcommand{\R}{\mathbf{R}}
\newcommand{\N}{\mathbf{N}}

\begin{document}

\title{Arcs intersecting at most once}

\author[P.~Przytycki]{Piotr Przytycki$^{\dag}$}
\address{Dep.\ of Math.\ \& Stat., McGill University\\
Burnside Hall, Room 1005, 805 Sherbrooke St.\ W\\
Montreal, Quebec, Canada H3A 2K6\\
and Inst.\ of Math., Polish Academy of Sciences\\
 \'Sniadeckich 8, 00-656 Warsaw, Poland}
\email{pprzytyc@mimuw.edu.pl}
\thanks{$\dag$ Partially supported by National Science Centre DEC-2012/06/A/ST1/00259 and NSERC}

\maketitle

\begin{abstract}
\noindent
We prove that on a punctured oriented surface with Euler characteristic $\chi<0$, the maximal cardinality of a set of essential simple arcs that are pairwise non-homotopic and intersecting at most once is $2|\chi|(|\chi|+1)$. This gives a cubic estimate in $|\chi|$ for a set of curves pairwise intersecting at most once on a closed surface. We also give polynomial estimates in $|\chi|$ for sets of arcs and curves pairwise intersecting a uniformly bounded number of times. Finally, we prove that on a punctured sphere the maximal cardinality of a set of arcs starting and ending at specified punctures and pairwise intersecting at most once is $\frac{1}{2}|\chi|(|\chi|+1)$.
\end{abstract}

\section{Introduction}
\label{sec:intro}

Let $S$ be a connected oriented punctured surface of finite type with Euler characteristic $\chi<0$. We start with the following observation.

\begin{rem}
\label{rem:disjoint arcs}
The maximal cardinality of a set of essential simple arcs on $S$ that are pairwise non-homotopic and disjoint is $3|\chi|$. This follows from the fact that every such set can be extended to form an ideal triangulation of the surface. Fixing any hyperbolic metric, the surface $S$ has area~$2\pi|\chi|$, while an ideal triangle has area~$\pi$. Thus there are $2|\chi|$ triangles, and these have $6|\chi|$ sides among which each arc appears twice.
\end{rem}

The main result of the article is a similar formula for arcs pairwise intersecting at most once.

\begin{thm}
\label{thm:arc}
The maximal cardinality of a set $\mathcal A$ of essential simple arcs on $S$ that are pairwise non-homotopic and intersecting at most once is $$f(|\chi|)=2|\chi|(|\chi|+1).$$
\end{thm}

\begin{ex}
\label{ex:arc} This bound is sharp. The surface $S$ can be cut
along disjoint arcs into an ideal polygon $P$. Since the Euler
characteristic of $P$ equals $1$, the number of arcs we have
cut along equals $1-\chi=|\chi|+1$, so that $P$ is a
$(2|\chi|+2)$--gon. We consider the union $\mathcal A$ of the
set of all diagonals of $P$, which has cardinality
$\frac{(2|\chi|+2)(2|\chi|-1)}{2}$, with the set of $|\chi|+1$
arcs we have cut along. In total,
$$|\mathcal A|=\frac{(2|\chi|+2)(2|\chi|-1)}{2}+(|\chi|+1)=(|\chi|+1)\big((2|\chi|-1)+1\big).$$
\end{ex}

We deduce a (non-sharp) cubic bound on curves pairwise intersecting at most once. Let $g$ be the genus of $S$. We now allow $S$ to be closed. Let $f(\cdot)$ be the function from Theorem~\ref{thm:arc}.

\begin{thm}
\label{thm:curve} The cardinality of a set $\mathcal{C}$ of
nonperipheral essential simple closed curves on $S$ that are
pairwise non-homotopic and intersecting at most once is at most
$$g\cdot\big(2f(|\chi|)+1\big)+|\chi|-1.$$
\end{thm}

The question about the maximal cardinality of $\mathcal{C}$ was
asked by Farb and Leininger. Apart from the intrinsic beauty of
the problem, one motivation to compute this value is that it is
one more than the dimension of the \emph{$1$--curve complex},
see \cite[Quest 1]{S}. This complex is constructed by taking
the flag span of the graph whose vertices correspond to curves
on $S$ and edges join the vertices corresponding to curves
intersecting at most once, as opposed to disjoint curves for
the usual curve complex. The $1$--curve complex might have
better properties than the usual one (like the Rips complex
versus the Cayley complex of a hyperbolic group), in particular
it might be contractible \cite[Quest 2]{S} as opposed to the
usual curve complex \cite{H}. Another interpretation of the
condition on~$\mathcal{C}$, this time in the terms of the
mapping class group of $S$, is that Dehn twists around the
curves in $\mathcal{C}$ pairwise either commute or satisfy the
braid relation.

For $S$ a torus, $\mathcal{C}$ consists of at most $3$ curves.
For $S$ a closed genus 2 surface, the maximal cardinality of
$\mathcal{C}$ is 12, as proved by Malestein, Rivin, and Theran
\cite{MRT}. They also gave examples of $\mathcal{C}$ with size
quadratic in $|\chi|$ for arbitrary surfaces. Note that these
sets are not the sets of systoles (shortest curves) for any
hyperbolic metric as it was shown by Parlier that sizes of such
sets grow subquadratically \cite{P}. Until our work, the best
upper bound for the cardinality of $\mathcal{C}$ was only
exponential \cite{MRT}. Farb and Leininger obtained the lower
quadratic and upper exponential bounds as well \cite{L}.

The question of Farb and Leininger is a particular case of a
problem studied by Juvan, Malni\v{c} and Mohar \cite{JMM}, who
allowed the curves to intersect at most a fixed number of times
$k$. They proved that there exists an upper bound on the
cardinality of $\mathcal{C}$ if the surface is fixed. We
establish upper bounds polynomial in $|\chi|$ for the sets of
either curves or arcs, where for the arcs we give also lower
bounds of the same degree. We do not have, even for arcs and
$k=2$, a guess of an explicit formula.

\begin{thm}
\label{thm:multi}
The maximal cardinality of a set $\mathcal A$ of essential simple arcs on $S$ that are pairwise non-homotopic and intersecting at most $k$ times grows as a polynomial of degree $k+1$ in $|\chi|$.
\end{thm}
\begin{cor}
\label{cor:multi}
The maximal cardinality of a set $\mathcal C$ of nonperipheral essential simple closed curves on $S$ that are pairwise non-homotopic and intersecting at most $k$ times
grows at most as a polynomial of degree $k^2+k+1$ in $|\chi|$.
\end{cor}

Note that in Corollary~\ref{cor:multi} the upper bound does not
match Aougab's recent lower bound, which is polynomial of
degree $\lfloor \frac{k+1}{2}\rfloor+1$ \cite{A}, improving the
degree $\frac{k}{4}$ from \cite{JMM}. Thus in
Corollary~\ref{cor:multi}, similarly as in
Theorem~\ref{thm:curve}, there is still room for improvement.
Our upper bound is also not efficient if we fix the surface and
vary $k$, compared to \cite{JMM}.

Finally, we prove the following variation of Theorem~\ref{thm:arc}.

\begin{thm}
\label{thm:punctured_main} Let $S$ be a punctured sphere, and
let $p,p'$ be its (not necessarily distinct) punctures. Then
the maximal cardinality of a set $\mathcal{A}$ of essential
simple arcs on $S$ that are starting at $p$ and ending at $p'$,
and pairwise intersecting at most once is
$$\frac{1}{2}|\chi|(|\chi|+1).$$
\end{thm}

\textbf{Organisation.}
The outline of the proof of Theorem~\ref{thm:arc} and the deduction of Theorem~\ref{thm:curve} is given in Section~\ref{sec:nibs}. The two key lemmas are proved in Section~\ref{sec:slits}. In Section~\ref{sec:multi} we prove Theorem~\ref{thm:multi} and Corollary~\ref{cor:multi}.
We prove Theorem~\ref{thm:punctured_main} in Section~\ref{sec:punctured}.

\medskip

\textbf{Acknowledgement.} I thank Chris Leininger for the
discussions on this problem during my stay in Urbana-Champaign,
and for his comments on the first draft of the article. I thank
Marcin Sabok for the short proof of Lemma~\ref{lem_n_points}. I
also thank Igor Rivin and the referees for all the remarks and
corrections that allowed to improve the article.

\section{Nibs}
\label{sec:nibs}

In this Section, we give the proof of Theorem~\ref{thm:arc}, up
to two lemmas postponed to Section~\ref{sec:slits}. Then we
deduce Theorem~\ref{thm:curve}. A rough idea is that we
generalise the argument of Remark~\ref{rem:disjoint arcs}. We
replace disjoint triangles of the triangulation with
\emph{nibs}, which have also area $\pi$ but might
(self-)intersect. However, we have a linear bound of the number
of nibs at a point (Proposition~\ref{prop:nib_overlap}).

As in the introduction, $S$ is an oriented punctured surface
with Euler characteristic $\chi<0$. An \emph{arc} is a map
$(-\infty,\infty)\rightarrow S$ that converges to punctures at
infinities. An arc is \emph{simple} if it is embedded and
\emph{essential} if it is not homotopic into a puncture
neighbourhood (called a \emph{cusp}). Unless stated otherwise,
all arcs are simple and essential and sets of arcs consist of
arcs that are pairwise non-homotopic. Assume that we have a set
$\mathcal{A}$ of arcs on $S$. We fix an arbitrary complete
hyperbolic metric on $S$ and assume that all arcs in
$\mathcal{A}$ are realised as geodesics. In particular they are
pairwise in \emph{minimal position}, that is minimising the
number of intersection points in their homotopy classes.

\begin{defin}
\label{def:tip} A \emph{tip} $\tau$ of $\mathcal{A}$ is a pair
$(\alpha,\beta)$ of oriented arcs in $\mathcal{A}$ starting at
the same puncture and consecutive. That is to say that there is
no other arc in $\mathcal{A}$ issuing from this puncture in the
clockwise oriented cusp sector from $\alpha$ to $\beta$.

Let $\tau=(\alpha,\beta)$ be a tip and let $N_\tau$ be an open
abstract ideal hyperbolic triangle with vertices $a,t,b$. The
tip $\tau$ determines a unique local isometry $\nu_\tau\colon
N_\tau\rightarrow S$ sending $ta$ to $\alpha$, $tb$ to $\beta$
and mapping a neighbourhood of $t$ to the clockwise oriented
cusp sector from $\alpha$ to $\beta$. We call $\nu_\tau$ the
\emph{nib} of $\tau$.
\end{defin}

\begin{figure}
\begin{center}
\includegraphics[width=0.35\textwidth]{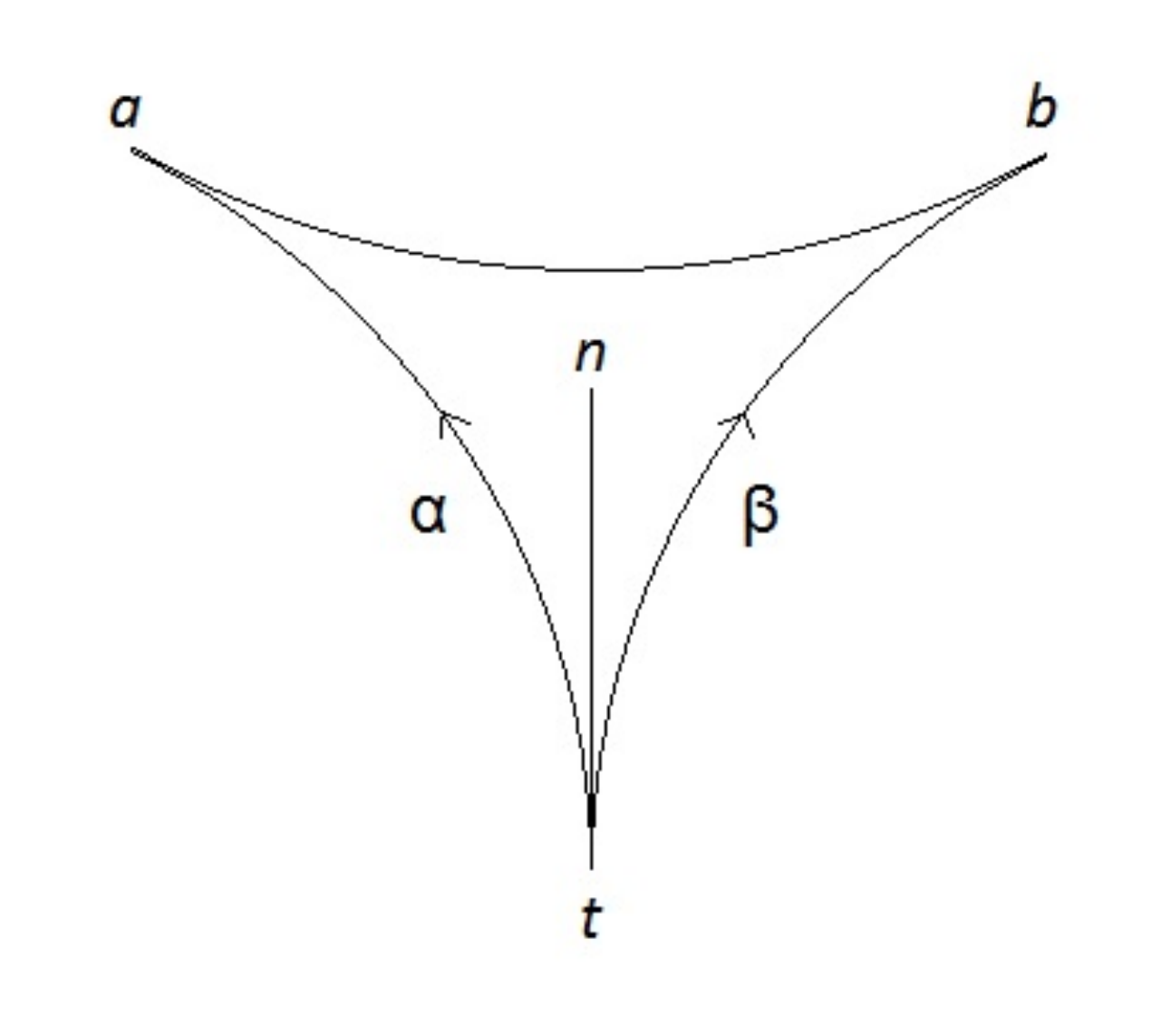}
\end{center}
\caption{A tip, its nib and a slit}
\label{fig:nib}
\end{figure}

\begin{prop}
\label{prop:nib_overlap} Suppose that the arcs in $\mathcal{A}$
pairwise intersect at most once. Let $\mathcal{\nu}\colon
N=\bigsqcup_\tau N_\tau\rightarrow S$ be the disjoint union of
all the nibs $\nu_\tau$. Then for each $s\in S$ the preimage
$\nu^{-1}(s)$ consists of at most $2(|\chi|+1)$ points.
\end{prop}

Here are the ingredients of the proof of Proposition~\ref{prop:nib_overlap}.

\begin{lemma}
\label{lem:arc_disjoint} Suppose that we have a partition of
the set of punctures of $S$ into $\mathcal{P}_1$ and
$\mathcal{P}_2$. The maximal cardinality of a set of pairwise
non-homotopic disjoint essential simple arcs that start in
$\mathcal{P}_1$ and end in $\mathcal{P}_2$ is $2|\chi|$.
\end{lemma}
\begin{proof}
Without loss of generality we can assume that the set of arcs is maximal. Then all the complementary components are ideal squares, so that each of them has area $2\pi$. Since the area of $S$ is $2\pi|\chi|$, there are $|\chi|$ such squares. Each square has $4$ sides and each arc is a side of two squares. Thus there are $2|\chi|$ arcs.
\end{proof}

\begin{defin}
\label{def:slit} Let $n\in N_\tau$ be a point in the domain of
a nib. The \emph{slit} at $n$ is the restriction of $\nu_\tau$
to the geodesic ray in $N_\tau$ joining $t$ with $n$. See
Figure~\ref{fig:nib}.
\end{defin}

By a \emph{geodesic ray} we will always mean a geodesic joining a point at infinity or a puncture with a point in the interior of $\H^2$ or $S$.
We postpone the proofs of the following two lemmas to Section~\ref{sec:slits}.

\begin{lemma}
\label{lem:slit_embedded}
A slit is an embedding.
\end{lemma}

Actually, it can be proved that if $\alpha$ and $\beta$ are
disjoint, then the entire nib embeds. Otherwise it is at worst
$2$ to $1$. We emphasise that in Lemma~\ref{lem:slit_embedded}
we do not assume that the arcs in $\mathcal{A}$ pairwise
intersect at most once.

\begin{lemma}
\label{lem:slit_disjoint}
Suppose that the arcs in $\mathcal{A}$ pairwise intersect at most once. If for distinct $n,n'\in N$ we have $\nu(n)=\nu(n')$, then the images in $S$ of the slits at $n,n'$ are disjoint except at the endpoint.
\end{lemma}

\begin{proof}[Proof of Proposition~\ref{prop:nib_overlap}]
Let $S'$ be the surface obtained from $S$ by introducing an
additional puncture at $s$. Then the absolute value of the
Euler characteristic of $S'$ is $|\chi|+1$. Consider the
collection of all slits at preimages of $s$, and let
$\mathcal{S}$ be the set of corresponding arcs on $S'$. The
arcs in $\mathcal{S}$ are simple by
Lemma~\ref{lem:slit_embedded} and pairwise disjoint by
Lemma~\ref{lem:slit_disjoint}. The set of punctures of $S'$
decomposes into $\mathcal{P}_1=\{s\}$ and $\mathcal{P}_2$,
which are the punctures inherited from the punctures of $S$.
Each arc in $\mathcal{S}$ joins $\mathcal{P}_1$ with
$\mathcal{P}_2$, hence by Lemma~\ref{lem:arc_disjoint} we have
$|\mathcal{S}|\leq 2(|\chi|+1)$.
\end{proof}

\begin{proof}[Proof of Theorem~\ref{thm:arc}]
Each arc in $\mathcal{A}$ is the first arc of exactly two tips, depending on its oriention, so the area of $N$ equals $2|\mathcal{A}|\pi$. The area of $S$ equals $2\pi|\chi|$. By Proposition~\ref{prop:nib_overlap}, the map $\nu\colon N\rightarrow S$ is at most $2(|\chi|+1)$ to $1$, so we have $$2|\mathcal{A}|\pi\leq 2\pi|\chi|\cdot 2(|\chi|+1).$$
\end{proof}

\begin{quest}
Example~\ref{ex:arc} does not exploit all configurations of
${2}|\chi|(|\chi|+1)$ arcs pairwise intersecting at most once.
An example of another configuration is obtained by viewing the
four-punctured sphere $S$ as the boundary of a tetrahedron $T$
with vertices removed. Consider the following set of 12 arcs on
$S$. The first 6 arcs are the $6$ edges of $T$. Each of the
other $6$ arcs is obtained by taking the midpoint $m$ of one of
the edges and concatenating at $m$ the medians of the two faces
of $T$ containing~$m$.

What can one say about the space of all configurations of ${2}|\chi|(|\chi|+1)$ arcs pairwise intersecting at most once?
\end{quest}

We now deduce Theorem~\ref{thm:curve}. A closed curve on $S$ is
\emph{simple} if it is embedded, \emph{essential} if it is
homotopically nontrivial, and \emph{nonperipheral} if it is not
homotopic into a cusp. In what follows, unless stated
otherwise, all curves are closed, simple, essential,
nonperipheral and sets of curves consist of curves that are
pairwise non-homotopic. We fix an arbitrary hyperbolic metric
on $S$ and realise all curves as geodesics, so that they are
pairwise in minimal position.

\begin{proof}[Proof of Theorem~\ref{thm:curve}]
Note that if all the curves in $\mathcal{C}$ are separating,
then they are disjoint. The theorem follows since a set of
disjoint separating curves has cardinality bounded by
$|\chi|-1$. Henceforth we will assume that there is a
non-separating curve in $\mathcal{C}$.

We prove the bound by induction on $g$ assuming $\chi$ is
fixed. If $g=0$, then all curves are separating, a case which
we have already discussed. If $g>0$, choose a non-separating
curve $\alpha$ in $\mathcal{C}$. Then $\mathcal{C}-\{\alpha\}$
partitions into the set $\mathcal{C}_\mathrm{dis}$ of curves
disjoint from $\alpha$ and the set $\mathcal{C}_\mathrm{int}$
of curves intersecting $\alpha$. Let $S'$ be the punctured
surface obtained from $S$ by cutting along $\alpha$. Note that
the Euler characteristic of $S'$ equals the Euler
characteristic $\chi$ of $S$ and the genus of $S'$ is one less
than that of $S$. By the inductive hypothesis we have
$|\mathcal{C}_\mathrm{dis}|\leq (g-1)(2f(|\chi|)+1)+|\chi|-1$.

Each curve in
$\mathcal{C}_\mathrm{int}$ is cut to an arc on $S'$. Two different curves cut to the same arc differ by a power of the Dehn twist $D$ along $\alpha$. Since curves differing by $D^l$ intersect in $l$ points, there might be at most $2$ such curves in $\mathcal{C}_\mathrm{int}$ for each arc on $S'$.
Thus by Theorem~\ref{thm:arc} we have $|\mathcal{C}_\mathrm{int}|\leq 2f(|\chi|)$. Hence
\begin{align*}
|\mathcal{C}|=1+|\mathcal{C}_\mathrm{dis}|+|\mathcal{C}_\mathrm{int}|&\leq 1+(g-1)(2f(|\chi|)+1)+|\chi|-1 + 2f(|\chi|)=\\
&=g\cdot(2f(|\chi|)+1)+|\chi|-1.
\end{align*}
\end{proof}

\section{Slits}
\label{sec:slits}

In this section we complete the proof of Theorem~\ref{thm:arc}, by proving Lemmas~\ref{lem:slit_embedded} and~\ref{lem:slit_disjoint} about slits.

\begin{proof}[Proof of Lemma~\ref{lem:slit_embedded}]
We identify the universal cover of $S$ with the hyperbolic
plane $\H^2$ and $\pi_1S$ with the deck transformation group.
Orient the circle at infinity. For points $x,x'$ at infinity we
will use the notation $[x,x']$ for the oriented sub-interval of
the circle at infinity, and notation $xx'$ for the geodesic arc
inside $\H^2$. Choose a lift of $\nu_\tau$ to $\H^2$, and
identify its image with $N_\tau$. Extend the slit to a ray
$\gamma$ joining $t$ with the side $ab$ of $N_\tau$. We need to
show that the interior of $\gamma$ embeds in $S$.

\begin{figure}
\begin{center}
\includegraphics[width=0.6\textwidth]{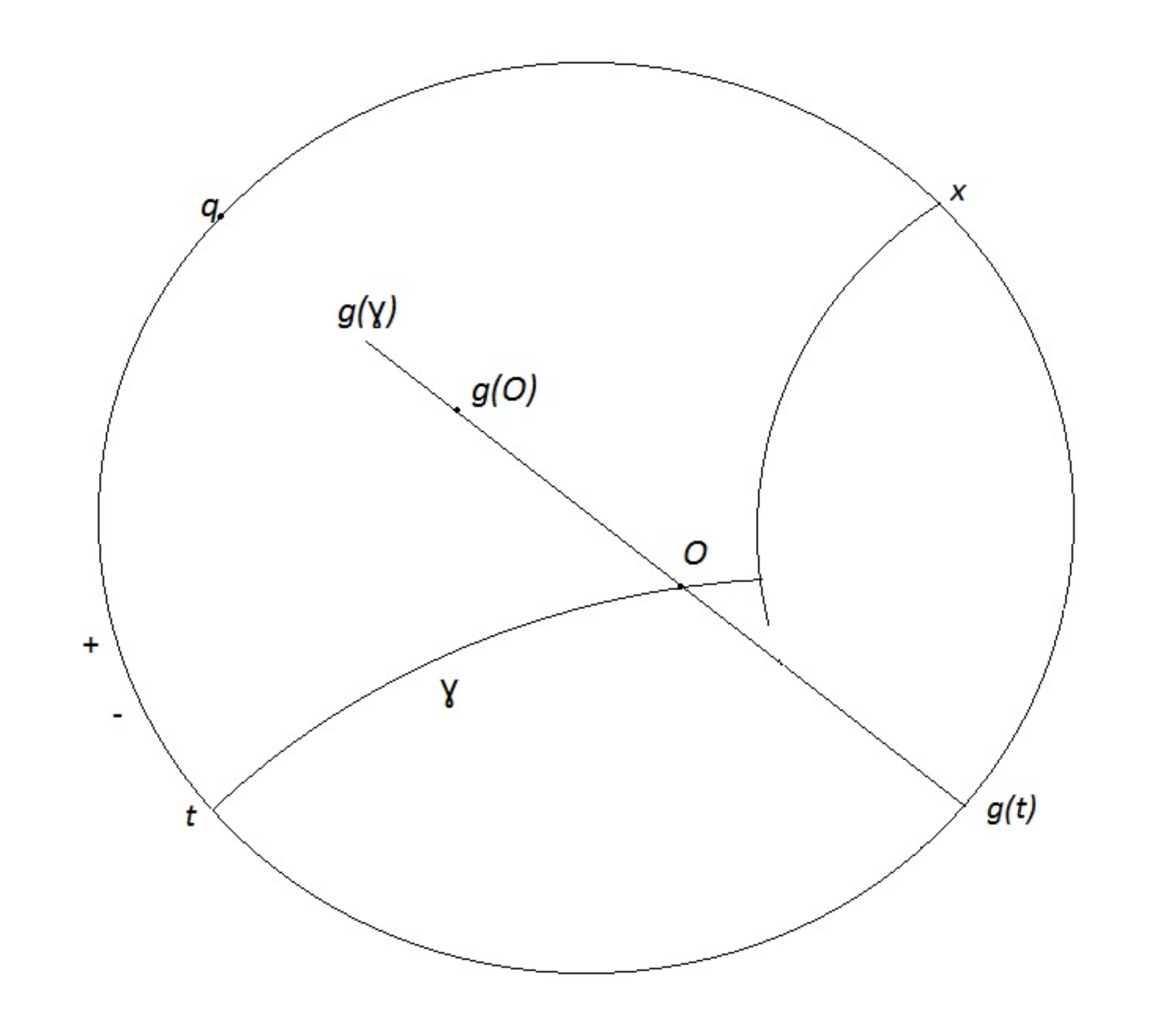}
\end{center}
\caption{Here we must take $x=a$, since $b$ lies on the wrong side of $g(t)q$}
\label{fig:lem1}
\end{figure}

Suppose that $\gamma$ intersects its translate $g(\gamma)$ in a
point $O$ for some nontrivial $g\in\pi_1S$. Without loss of
generality assume that $g(O)$ is farther from $g(t)$ than $O$
on $g(\gamma)$. Let $q$ be the point at infinity that is the
endpoint of the geodesic extension of $g(\gamma)$. See
Figure~\ref{fig:lem1}. We analyse where are the fixed point(s)
of $g$. Since $O$ and $g(O)$ lie in that order on $g(t)q$, in
the case where $g$ is hyperbolic both fixed points lie on one
side of $g(t)q$, and on that side the repelling point separates
the attracting point from $g(t)$. Hence $t$ also has to lie on
that side, between the repelling point and $g(t)$. Without loss
of generality assume that $t,g(t),q$ are cyclically ordered. We
conclude that the fixed points lie in the interval $[q,t]$.
Moreover, $g$ restricts on the interval $[g(t),q]$ to an
increasing function into $[g(t),t]$. The same holds in the case
where $g$ is parabolic.

There is $x\in \{a,b\}$ that lies in the interval $[g(t),q]$. Then the points $t, g(t), x, g(x)$ lie at the circle at infinity in that cyclic order. This means that the projection of the arc $tx$ to $S$ self-intersects, which is a contradiction with the fact that it belongs to $\mathcal{A}$.
\end{proof}

In the proof of Lemma~\ref{lem:slit_disjoint} we will need the following. We parametrise a ray $(-\infty,0]\rightarrow S$ so that the limit at $-\infty$ is a puncture and the image of $0$ is a point on $S$.

\begin{sublemma}
\label{sublem:intersection}
Let $H,H'\colon(-\infty, 0]\times[-1,1]$ be homotopies of rays on $S$, such that
\begin{itemize}
\item
$H(0,\cdot)=H'(0,\cdot)$,
\item
arcs $H(\cdot, y)$ and $H'(\cdot,y)$ are in minimal position on the surface $S-H(0,y)$ for $y=-1,1$,
\item
rays $H(\cdot, y)$ and $H'(\cdot,y)$ are embedded for $y\in [-1,1]$.
\end{itemize}
Then the arcs $H(\cdot, -1)$ and $H'(\cdot, -1)$ on $S-H(0,-1)$ intersect the same number of times as the arcs
$H(\cdot, 1)$ and $H'(\cdot, 1)$ on $S-H(0,1)$.
\end{sublemma}
\begin{proof}
We identify (up to homotopy) the surfaces $S-H(0,y)$ for $y\in [-1,1]$ by pushing the puncture $H(0,y)=H'(0,y)$. The last condition, saying that all the rays are embedded, ensures that the arcs $H(\cdot, y)$ (respectively $H'(\cdot,y)$) on $S-H(0,y)$ can be identified. The assertion follows from the fact that the number of intersection points between arcs in minimal position for $y=-1,1$ is a homotopy invariant.
\end{proof}

\begin{proof}[Proof of Lemma~\ref{lem:slit_disjoint}]
Let $\tau,\tau'$ be tips with distinct $n\in N_\tau, \ n'\in
N_{\tau'}$ and $\nu(n)=\nu(n')$. Note that by the requirement
in Definition~\ref{def:tip} of a tip that $\alpha$ and $\beta$
are consecutive, the slits at $n,n'$ are distinct rays on $S$.
Choose lifts of $N_\tau, N_{\tau'}$ to $\H^2$ so that $n$ is
identified with $n'$. Extend the slits at $n,n'$ to geodesic
rays $\gamma, \gamma'$ from $t,t'$ to $ab, a'b'$. Without loss
of generality suppose that the interval $(t,t')$ at infinity
contains the endpoints of the geodesic extensions of $\gamma,
\gamma'$.

First consider the case where there are $x\in \{a,b\}, x'\in
\{a',b'\}$ that lie in $(t,t')$ so that $x'< x$, see the left
side of Figure~\ref{fig:lem2}. Then the geodesic arcs $tx,
t'x'$ intersect at a point $m$. Let $H(\cdot,\cdot)$ be the
projection to $S$ of the homotopy of geodesic rays joining $tn$
to $tm$ ending in the geodesic segment $nm$. By
Lemma~\ref{lem:slit_embedded}, these rays embed in $S$.
Similarly let $H'(\cdot,\cdot)$ be the projection of the
homotopy joining $t'n$ to $t'm$. Since the rays are geodesic,
the arcs $H(\cdot,y),H'(\cdot,y)$ are in minimal position on
$S-H(0,y)$. By Sublemma~\ref{sublem:intersection}, if the slits
at $n,n'$ intersect outside the endpoint, then the projections
to $S$ of $tx$ and $t'x'$ intersect at least once outside the
projection of $m$, hence at least twice in total, which is a
contradiction with the definition of~$\mathcal{A}$.

The other case is that $a,b'\in [t',t]$ and $b\leq a'\in (t,t')$, see the right side of Figure~\ref{fig:lem2}. If $a=t'$ or $b'=t$, then we have also the second equality, since $\alpha$ and $\beta$ were required to be consecutive in Definition~\ref{def:tip} of a tip. In this situation, let $m$ be any point on the geodesic arc $tt'$. Construct $H,H'$ as before. By Sublemma~\ref{sublem:intersection}, if the slits at $n,n'$ intersect outside the endpoint, then the projections to $S$ of $tm$ and $t'm$ intersect outside the projection of $m$. This contradicts the fact that the projection of $tt'=ta$ is in $\mathcal{A}$, hence simple.

\begin{figure}
\begin{center}
\includegraphics[width=\textwidth]{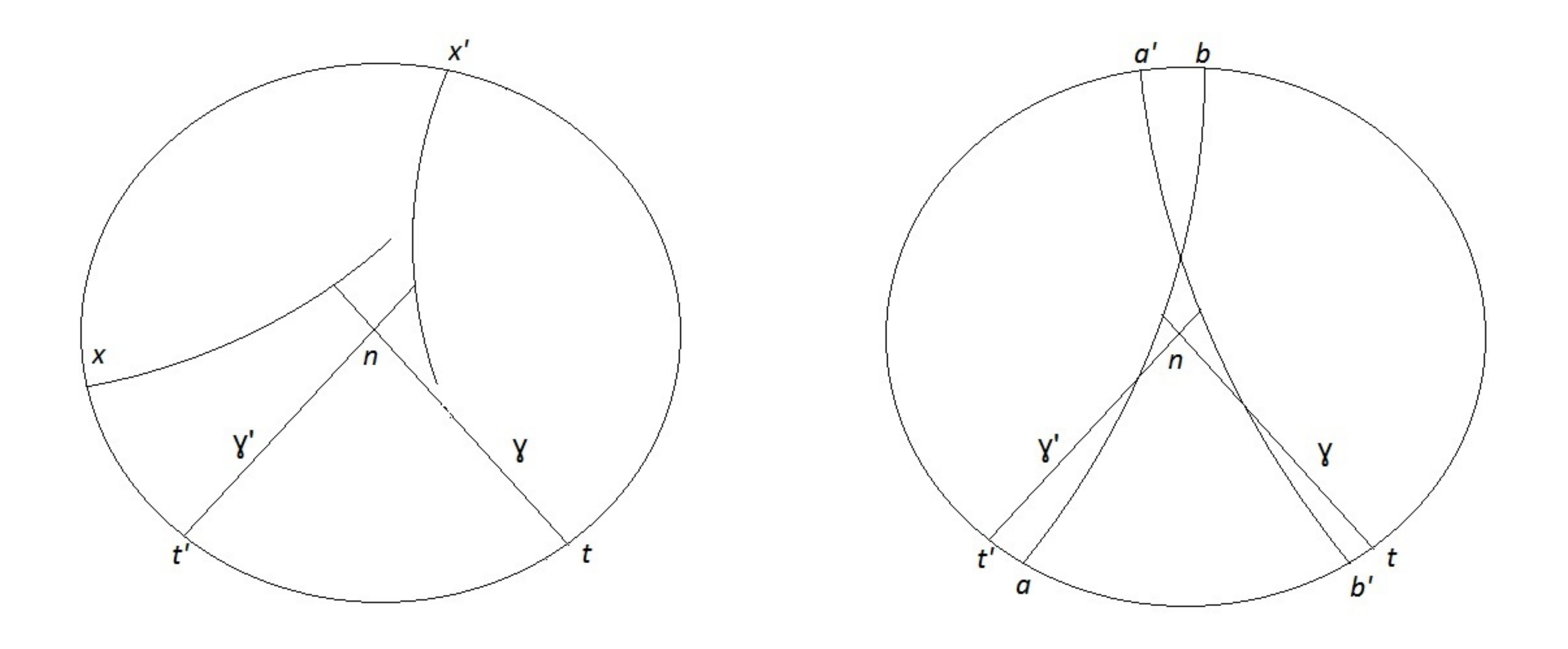}
\end{center}
\caption{Instance of the first case with $x=a, x'=a'$ on the left, and the second case on the right}
\label{fig:lem2}
\end{figure}

Thus we can assume $t'<a<b'<t$. Then the geodesic arcs $ta, t'b'$ intersect. We denote their intersection point by $m$ and apply Sublemma~\ref{sublem:intersection} as in the first case.
\end{proof}

The same proof gives also the following.

\begin{lemma}
\label{lem:multi}
Suppose that the arcs in $\mathcal{A}$ pairwise intersect at most $k\geq 1$ times. If for distinct $n,n'\in N$ we have $\nu(n)=\nu(n')$, then the images in $S$ of slits at $n,n'$ intersect at most $k-1$ times outside the endpoint.
\end{lemma}

\begin{quest}
A slit can be viewed as an interpolation between the arcs $ta$ and $tb$ of a nib. Is there an alternate proof of Lemmas~\ref{lem:slit_embedded} and~\ref{lem:slit_disjoint} using convexity of the distance function?
\end{quest}

\section{Multiple intersections}
\label{sec:multi}

In this Section we prove Theorem~\ref{thm:multi} and Corollary~\ref{cor:multi}, starting with the lower bound in Theorem~\ref{thm:multi}.

\begin{ex}
\label{ex:multi}
Consider a twice punctured sphere with punctures $p,p'$ viewed as a cylinder. Let $S$ be the surface obtained by putting on each of $k+1$ concentric circles separating $p$ from $p'$ a set of $\frac{|\chi|}{k+1}$ punctures. Let $\alpha$ be an arc joining $p$ to $p'$ crossing each of these concentric circles exactly once.
Consider the set $\mathcal{A}$ of arcs joining $p$ to $p'$ crossing each of these concentric circles exactly once, and disjoint from $\alpha$. Arcs in $\mathcal{A}$ pairwise intersect at most $k$ times, since there might be at most one intersection point in each sub-cylinder bounded by consecutive concentric circles. Each arc in $\mathcal{A}$ is determined by the locations of its intersection points with concentric circles, for each of which there are $\frac{|\chi|}{k+1}+1$ possibilities. Moreover, these locations determine the arc uniquely, except that an arc homotopic to $\alpha$ will be given two times, by the two extreme sets of locations. Thus $$|\mathcal{A}|=\Big(\frac{|\chi|}{k+1}+1\Big)^{k+1}-1.$$
\end{ex}

\begin{proof}[Proof of Theorem~\ref{thm:multi}]
By Example~\ref{ex:multi} it remains to prove the upper bound. We prove it by induction on the number $k$ of allowed intersections. Case $k=0$ is discussed in Remark~\ref{rem:disjoint arcs}. The inductive step from $k=0$ to $k=1$ has been performed in the proof of Theorem~\ref{thm:arc}, and we now generalise it to any $k$. Suppose that we have proved that the maximal cardinality of a set of arcs on $S$ pairwise intersecting at most $k\geq 0$ times is $\leq C(k)|\chi|^{k+1}$ for some constant $C(k)$ depending only on $k$. Let $\mathcal{A}$ be a set of arcs on $S$ pairwise intersecting at most $k+1$ times. They give rise to $2|\mathcal{A}|$ tips. Let $s\in S$. By Lemmas~\ref{lem:slit_embedded} and~\ref{lem:multi} the slits at the preimage $\nu^{-1}(s)$ in the union of all the nibs form a set of simple arcs on $S-s$ pairwise intersecting at most $k$ times. Thus by the inductive assumption the map $\nu$ is at most $C(k)(|\chi|+1)^{k+1}$ to $1$. Hence the area $2|\mathcal{A}|\pi$ of the domain of $\nu$ is bounded by
$$2\pi|\chi|\cdot C(k)(|\chi|+1)^{k+1}\leq 2\pi C(k+1)|\chi|^{k+2},$$ for some $C(k+1)$, as desired.
\end{proof}

\begin{proof}[Proof of Corollary~\ref{cor:multi}]
We orient all the curves. We cut the surface along an arbitrary curve $\alpha\in \mathcal{C}$, possibly separating. The surface $S-\alpha$ is either connected of Euler characteristic $\chi$ or consists of two components with Euler characteristics $\chi_1,\chi_2$ satisfying $\chi_1+\chi_2=\chi$, hence $|\chi_1|^{k+1}+|\chi_2|^{k+1}<|\chi|^{k+1}$.

Let $\mathcal{C}_{\mathrm{int}}$ denote the set of curves in
$\mathcal{C}$ intersecting $\alpha$. Each curve in
$\mathcal{C}_{\mathrm{int}}$ splits into $\leq k$ oriented arcs
on $S-\alpha$. These arcs pairwise intersect at most $k$ times,
hence by Theorem~\ref{thm:multi} there is at most
$C(k)|\chi|^{k+1}$ of them. Each curve in
$\mathcal{C}_{\mathrm{int}}$ can be reconstructed from an
ordered $k'$--tuple of $k'\leq k$ of these arcs, so this gives
$\leq k\big(C(k)|\chi|^{k+1}\big)^k$ possibilities, up to a
power $D^l$ of the Dehn twist $D$ around $\alpha$ with $l\leq
k$. (Note that we do not need to consider partial twists since
the $k'$--tuple of arcs is ordered.) Altogether this yields
$|\mathcal{C}_{\mathrm{int}}|\leq C'(k)|\chi|^{k(k+1)}$.

To bound $|\mathcal{C}-\mathcal{C}_{\mathrm{int}}|$, we need to continue performing the cutting procedure. The number of cutting steps is bounded by the cardinality of the maximal set of disjoint curves on $S$, which equals $|\chi|-1+g<2|\chi|.$ Thus we obtain a polynomial bound of degree $1+k(k+1)$ in $|\chi|$ for $|\mathcal{C}|$.
\end{proof}

\section{Punctured spheres}
\label{sec:punctured}

In this section we prove Theorem~\ref{thm:punctured_main}. We
start with giving the lower bounds. The construction will be
different depending on whether the ending puncture $p'$ is the
same or distinct from the starting puncture $p$.

\begin{ex}
\label{ex:punctured1} Let $S$ be a punctured sphere. We
construct a set of $\frac{1}{2}|\chi|(|\chi|+1)$ arcs on $S$
that are starting and ending at a specified puncture $p$,
pairwise intersecting at most once.

View $S$ as a punctured disc with the outer boundary corresponding to $p$ and other $|\chi|+1$ punctures lying on a circle $c$ parallel to the outer boundary, dividing $c$ into $|\chi|+1$ segments. Up to homotopy, there are $|\chi|+1$ rays joining the centre of the disc with $p$ and intersecting $c$ only once, at $|\chi|+1$ possible segments of $c$. Consider ${{|\chi|+1}\choose 2}$ arcs obtained by merging a pair of such rays. Up to homotopy they pairwise intersect at most once, at the centre of the disc.
\end{ex}

\begin{ex}
\label{ex:puncture2} Let $S$ be a punctured sphere. We
construct a set $\mathcal{A}$ of $\frac{1}{2}|\chi|(|\chi|+1)$
arcs on $S$ that are starting at a specified puncture $p$ and
ending at a distinct specified puncture $p'$, pairwise
intersecting at most once.

Consider a cylinder with boundary components corresponding to $p,p',$ and place other $|\chi|$ punctures on an arc $\alpha$ joining $p$ with $p'$ dividing it into an ordered set of $|\chi|+1$ segments. There are ${{|\chi|+1}\choose 2}$ choices of a pair $\alpha_1<\alpha_2$ of such segments. For each such pair we construct the following arc in $\mathcal{A}$. We start at the starting point of $\alpha$ and follow it on its left avoiding the punctures until we find ourselves on the level of $\alpha_1$. Then we cross $\alpha_1$ and change to the right side of $\alpha$. We stay on the right side until we reach $\alpha_2$, where we switch again to the left side and continue till the ending point of $\alpha$. It is easy to verify that arcs in $\mathcal{A}$ pairwise intersect at most once.
\end{ex}

In the proof of the upper bound in Theorem~\ref{thm:punctured_main} we need the following.

\begin{lemma}
\label{lem_n_points}
Suppose that we have $l$ points on the unit circle in $\R^2$ and a set $\mathcal I$ of pairwise intersecting chords between them. We allow degenerate chords that are single points. Then $|\mathcal{I}|\leq l$.
\end{lemma}
While this lemma might be well-known, we give a proof found by Marcin Sabok:
\begin{proof}
We label the points cyclically by the set $\{0,\ldots,l-1\}\subset \N$ and assign to each chord the interval of $\R$ with corresponding endpoints.
We keep the notation $\mathcal I$ for this set of intervals. By Helly's Theorem, there is a point $i\in \{0,\ldots,l-1\}$ belonging to all the intervals in $\mathcal I$. The centres of the intervals in $\mathcal I$ lie in the intersection of $\frac{1}{2}\N$ with $[\frac{i}{2},i+\frac{l-1-i}{2}]$, which has cardinality $l$. Thus if $|\mathcal I|\geq l+1$, then by the pigeon principle two intervals in $\mathcal I$ have the same centres. But this means that their corresponding chords are disjoint.
\end{proof}

\begin{proof}[Proof of Theorem~\ref{thm:punctured_main}]
By Examples~\ref{ex:punctured1} and~\ref{ex:puncture2} it
suffices to prove the upper bound, which we do by induction. We
first treat the case where the arcs in $\mathcal{A}$ are
required to start and end at the same puncture $p$.

For $|\chi|=1$ the theorem is trivial as on the
thrice-punctured sphere there is only one essential arc joining
$p$ with itself. Suppose that we have proved the theorem for
smaller $|\chi|$, in particular for the surface $\bar{S}$
obtained from $S$ by forgetting a puncture $r$ distinct from
$p$. For an arc $\alpha\in \mathcal{A}$ on $S$ let
$\bar{\alpha}$ be the corresponding possibly non-essential arc
on $\bar{S}$. Let $\bar{\mathcal{A}}$ denote the union of all
essential $\bar{\alpha}$. We will now analyse to what extent
this map $\mathcal{A}\rightarrow \bar{\mathcal{A}}$ is
well-defined and injective. To this end, we consider the
following \emph{exceptional} arcs in $\mathcal{A}$, and an
associated collection $\mathcal{I}$ of rays on $\bar{S}$ from
$p$ to $r$ and arcs from $p$ to $p$ passing through $r$.

The first instance of an \emph{exceptional} arc $\alpha$ is
when $\bar{\alpha}$ is non-essential. This happens if and only
if $\alpha$ separates $r$ from all other punctures. We include
in $\mathcal{I}$ the unique ray on $\bar{S}$ joining $p$ with
$r$ disjoint from $\alpha$.

Secondly, observe that if $\bar{\alpha}=\bar{\alpha}'$ for
$\alpha\neq\alpha'$, then since $\alpha$ and $\alpha'$
intersect at most once, they are in fact disjoint (up to
homotopy) and bound a bigon with one puncture $r$. There is no
third arc $\alpha''\in\mathcal{A}$ with
$\bar{\alpha}''=\bar{\alpha}$, since at most one component of
$S-{\alpha\cup \alpha'\cup\alpha''}$ contains $r$. We call such
$\alpha$ and $\alpha'$ \emph{exceptional} as well. We include
in $\mathcal{I}$ the unique essential arc on $\bar S$ joining
$p$ with $p$ and passing through $r$ that is disjoint from both
$\alpha$ and $\alpha'$.

Summarising, the map $\mathcal{A}\rightarrow \bar{\mathcal{A}}$
is well-defined and injective outside the set of exceptional
arcs of cardinality $|\mathcal{I}|$, where we avoid only one
exceptional arc from each pair bounding a bigon with puncture
$r$. By the inductive assumption we have
$|\bar{\mathcal{A}}|\leq \frac{1}{2}(|\chi|-1)|\chi|$, so it
suffices to prove $|{\mathcal{I}}|\leq |\chi|$.

\begin{figure}
\begin{center}
\includegraphics[width=\textwidth]{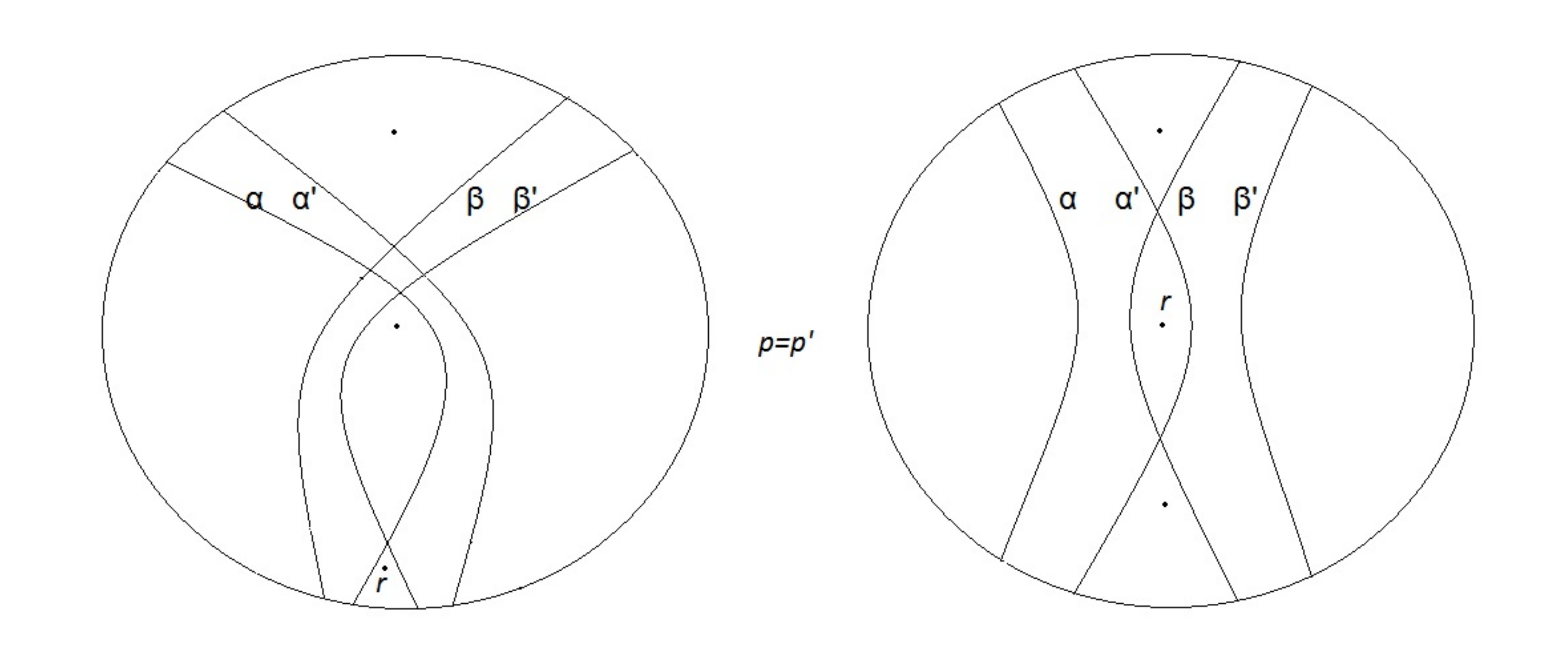}
\end{center}
\caption{On the left why the halves of the arcs in $\mathcal{I}$ are disjoint, on the right why $\mathcal{I}$ satisfies the hypothesis of Lemma~\ref{lem_n_points}}
\label{fig:bigons}
\end{figure}

Let $\mathcal{H}$ be the set of arcs on $S$ obtained from
${\mathcal{I}}$ by reintroducing the puncture $r$ and thus
splitting some arcs of ${\mathcal{I}}$ in half. Note that if
some halves coincide, we keep only one copy in $\mathcal{H}$.
Observe that the arcs in $\mathcal{H}$ are disjoint, since if
we had two intersecting halves of two arcs in $\mathcal{I}$,
then we could choose two corresponding exceptional arcs
intersecting at least twice. See the left side of
Figure~\ref{fig:bigons}. Arcs in $\mathcal{H}$ connect $r$ with
$p$, hence their complementary components are punctured bigons,
thus of area $\geq 2\pi$. Hence there are at most $|\chi|$ such
components and thus $|\mathcal{H}|\leq |\chi|$. We now
intersect $\mathcal{H}$ with a small circle centred at $r$.
Each arc in $\mathcal I$ is determined by a pair of points or a
point of this intersection, and we connect them by a chord.
These chords satisfy the hypothesis of
Lemma~\ref{lem_n_points}, since otherwise we could choose two
corresponding exceptional arcs intersecting at least twice. See
the right side of Figure~\ref{fig:bigons}. By
Lemma~\ref{lem_n_points} we have $|{\mathcal{I}}|\leq |\chi|$,
which finishes the proof in the case where $p'=p$.

We now consider the case where $p'\neq p$. The argument is the
same, with the following modifications. The puncture $r$ that
we are forgetting is required to be distinct from both $p$ and
$p'$. All $\bar{\alpha}$ are essential, since they connect
different punctures. However, it is no longer true that if
$\bar{\alpha}=\bar{\alpha}'$ for $\alpha\neq\alpha'$, then
$\alpha$ and $\alpha'$ are disjoint. Nevertheless, it is easy
to see that they may intersect only in the configuration
illustrated in Figure~\ref{fig:trigon}. In that case all arcs
in $\mathcal{A}$ are disjoint from arc $\alpha''$ from
Figure~\ref{fig:trigon}, thus without loss of generality we can
assume $\alpha''\in \mathcal{A}$. We then have
$\bar{\alpha}''=\bar{\alpha}$. Moreover, there are no other
arcs in $\mathcal{A}$ with the same $\bar{\alpha}$. We call
both $\alpha,\alpha'$ \emph{exceptional} and associate to them
the bigons that they bound with~$\alpha''$ containing~$r$.

\begin{figure}
\begin{center}
\includegraphics[width=0.45\textwidth]{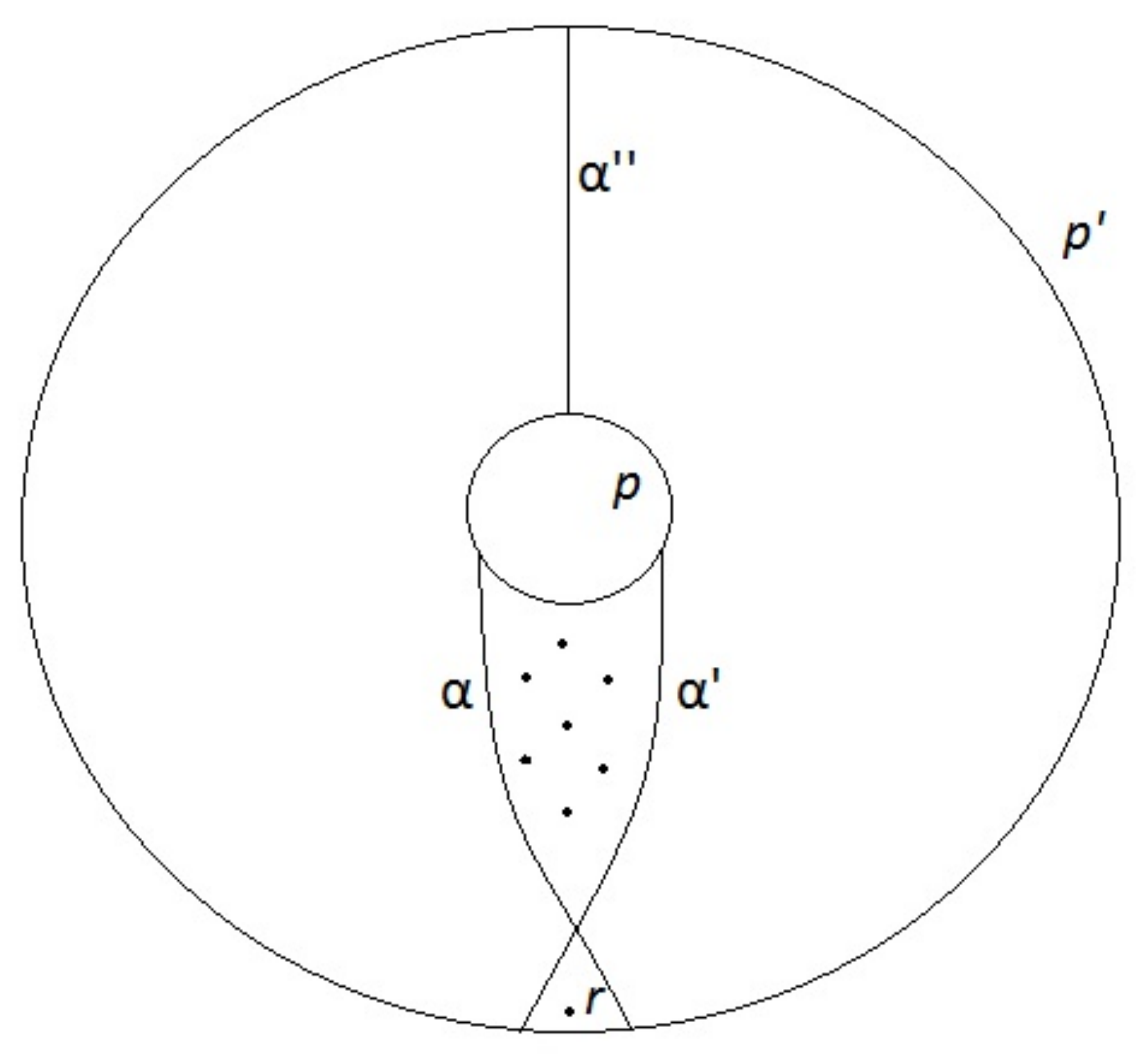}
\end{center}
\caption{The unique configuration of three arcs with common $\bar{\alpha}$}
\label{fig:trigon}
\end{figure}

We construct $\mathcal{I}$ and $\mathcal{H}$ as before.
However, as pointed out to us by one of the referees, among the
complementary components of $\mathcal{H}$, except for the
punctured bigons, there is a square, possibly with no
punctures, hence also only of area $\geq 2\pi$. Thus we can
only conclude with $|\mathcal{H}|\leq |\chi|+1$. We again
intersect $\mathcal{H}$ with a small circle centred at $r$. We
assign chords to arcs in $\mathcal I$ as before. Note that the
points of the intersection of the circle with $\mathcal{H}$,
depending on whether the arc in $\mathcal{H}$ joins $r$ to $p$
or $p'$, are partitioned in two families
$\mathcal{Q},\mathcal{Q'}$ of consecutive points. We consider
an additional chord joining the two outermost points of
$\mathcal{Q}$. Since each of the other chords connects a point
in $\mathcal{Q}$ to a point in $\mathcal{Q}'$, it intersects
the additional chord. Hence the chords satisfy the hypothesis
of Lemma~\ref{lem_n_points}. Thus $|\mathcal I|+1\leq
|\chi|+1$, as desired.
\end{proof}

\begin{quest}
Is there a way to make the proof of Theorem~\ref{thm:arc} work to prove Theorem~\ref{thm:punctured_main} or vice-versa?
\end{quest}


\begin{bibdiv}
\begin{biblist}

\bib{A}{article}{
   author={Aougab, Tarik},
   title={Constructing large $k$--systems on surfaces},
   journal={Topol. Appl.},
   date={2014}}

\bib{H}{article}{
   author={Harer, John L.},
   title={The virtual cohomological dimension of the mapping class group of
   an orientable surface},
   journal={Invent. Math.},
   volume={84},
   date={1986},
   number={1},
   pages={157--176}}

\bib{JMM}{article}{
   author={Juvan, Martin},
   author={Malni{\v{c}}, Aleksander},
   author={Mohar, Bojan},
   title={Systems of curves on surfaces},
   journal={J. Combin. Theory Ser. B},
   volume={68},
   date={1996},
   number={1},
   pages={7--22}}

\bib{L}{article}{
   author={Leininger, Christopher},
   date={2011},
   title={personal communication}}

\bib{MRT}{article}{
   author={Malestein, Justin},
   author={Rivin, Igor},
   author={Theran, Louis},
   title={Topological designs},
   journal={Geom. Dedicata},
   volume={168},
   date={2014},
   number={1},
   pages={221--233}}

\bib{P}{article}{
   author={Parlier, Hugo},
   title={Kissing numbers for surfaces},
   journal={J. Topol.},
   volume={6},
   date={2013},
   number={3},
   pages={777--791}}

\bib{S}{article}{
   title={A remark about the curve complex},
   author={Suoto, Juan}
   date={2007}
   eprint={http://www.math.ubc.ca/~jsouto/papers/d-curve.pdf}}

\end{biblist}
\end{bibdiv}
\end{document}